\documentclass[12pt]{amsart}
\usepackage{enumerate}
\usepackage{hyperref}
\usepackage{subfig}
\hypersetup{
    pdftitle=   {Unavoidable induced subgraphs in large graphs with no homogeneous sets},
   pdfauthor=  {Maria Chudnovsky, Ringi Kim, Sang-il Oum, and Paul Seymour}
}
\newcommand\abs[1]{\lvert #1\rvert}
\newcommand\MYSPACE{\hspace{2em}}
\usepackage{amsmath,tikz}
\newtheorem{THM}{Theorem}[section]

\newtheorem{COR}[THM]{Corollary}
\newtheorem{PROP}[THM]{Proposition}

\theoremstyle{remark}

\theoremstyle{definition}

\begin{document}
\title[Unavoidable induced subgraphs]{Unavoidable induced subgraphs in large graphs with no homogeneous sets}
\author{Maria Chudnovsky}
\address{Princeton University,Princeton, NJ 08544, USA}
\email{mchudnov@math.princeton.edu}
\author{Ringi Kim}
\address{Princeton University,Princeton, NJ 08544, USA}
\email{ringikim@princeton.edu}
\author{Sang-il Oum}
\address{Department of Mathematical Sciences, KAIST, Daejeon, 34141 South Korea}
\email{sangil@kaist.edu}
\author{Paul Seymour}
\address{Princeton University,Princeton, NJ 08544, USA}
\email{pds@math.princeton.edu}
\thanks{M.~C.~is supported by NSF grant DMS-1265803.
 S.~O.~is supported by Basic Science Research
  Program through the National Research Foundation of Korea (NRF)
  funded by  the Ministry of Science, ICT \& Future Planning
  (2011-0011653). 
  P.~S.~is supported by ONR grant N00014-10-1-0680 and 
  NSF grant DMS-1265563.
}
\date{April 20, 2015; revised December 12, 2015}
\begin{abstract}
A \emph{homogeneous} set of an $n$-vertex graph is a set $X$ of vertices  ($2\le \abs{X}\le n-1$) such that every vertex not in $X$ is either complete or anticomplete  to $X$. A graph is called \emph{prime} if it has no homogeneous set.
A \emph{chain} of length $t$ is a sequence of $t+1$ vertices such that for every vertex in the sequence except the first one, its immediate predecessor is its unique neighbor or its unique non-neighbor among all of its predecessors. 
We prove that for all $n$, there exists $N$ such that every prime graph with at least $N$ vertices contains one of the following graphs or their complements as an induced subgraph:
(1) the graph obtained from $K_{1,n}$ by subdividing every edge once, 
(2) the line graph of $K_{2,n}$,
(3) the line graph of the graph in (1),
(4) the half-graph of height $n$, 
(5) a prime graph induced by a chain of length $n$,  
(6) two particular graphs obtained from the half-graph of height $n$ by making one side a clique and adding one vertex.
\end{abstract}
\maketitle
\section{Introduction}\label{sec:intro}

All graphs in this paper are simple and undirected. 
We wish to prove a theorem analogous to the following theorems. (For missing definitions in the following, please refer to the referenced papers.)
\begin{description}
\item[Ramsey's theorem] \itshape \raggedright
Every sufficiently large graph contains $K_n$ or $\overline{K_n}$ as an induced subgraph.
\item[folklore; see Diestel~{\cite[Proposition 9.4.1]{Diestel2010}}] 
Every sufficiently large \textbf{connected} graph contains $K_n$, $K_{1,n}$, or a path of length $n$  as an induced subgraph.
\item[folklore; see Diestel~{\cite[Proposition 9.4.2]{Diestel2010}}] 
Every sufficiently large \textbf{$\mathbf 2$-connected} graph contains $C_n$ or $K_{2,n}$ as  a topological minor.
\item[Oporowski, Oxley, and Thomas~\cite{OOT1993}] Every sufficiently large \textbf{$\mathbf 3$-connected} graph contains the $n$-spoke wheel  or $K_{3,n}$ as  a minor.
\item[Oporowski, Oxley, and Thomas~\cite{OOT1993}] Every sufficiently large \textbf{internally $\mathbf 4$-connected} graph contains the $2n$-spoke double wheel, the $n$-rung circular ladder, the $n$-rung M\"obius ladder, or $K_{4,n}$ as  a minor.
\item [Ding, Chen~\cite{DC2004}] Every sufficiently large \textbf{connected and anticonnected} graph contains one of the following graphs or their complements as  an induced subgraph:  a path of length $n$, the graph obtained from $K_{1,n}$ by subdividing an edge once, $K_{2,n}$ minus one edge, or the graph obtained from $K_{2,n}$ by adding an edge between two degree-$n$ vertices $x_1$ and $x_2$ and adding a pendant edge at each $x_i$.
\item [Kwon, Oum~\cite{KO2014}] Every sufficiently large graph with \textbf{no non-trivial split} contains, as a vertex-minor, a cycle of length $n$ or the line graph of $K_{2,n}$.
\end{description}

These results state that every sufficiently large graph satisfying certain connectivity requirements contains at least one of the given graphs. Furthermore, in all these theorems, the list is best possible in the sense that each given graph satisfies the required connectivity, can grow its size arbitrary, and does not contain other given graphs in the list.

In this paper, we focus on graphs with no homogeneous sets. A set $X$ of vertices of a graph $G$ is \emph{homogeneous} if $2\le \abs{X}<\abs{V(G)}$ and every vertex outside of $X$ is either adjacent to all vertices in $X$ or adjacent to no vertex in $X$. In the literature, homogeneous sets are also called \emph{non-trivial modules}, \emph{partitive sets}, \emph{autonomous sets}, and various other terms~\cite{MR1984}. 
A graph is called \emph{prime} if it has no homogeneous set.

Homogeneous sets are widely used as a tool to study classes of graphs with respect to the induced subgraph relation.
If a graph $G$ is obtained from some non-trivial graph by substituting a non-trivial graph for a vertex, then $G$ has a homogeneous set and so $G$ is not prime. 
(A graph is \emph{non-trivial} if it has at least two vertices.) Thus, if a certain class $\mathcal C$ of graphs is closed under substitution
and a graph $H$ is not in $\mathcal C$ but all proper induced subgraphs of $H$ are in $\mathcal C$, then $H$ is prime.
Many important graph classes are closed under substitution.
For instance, Lov\'asz~\cite{Lovasz1972} proved that perfect graphs are closed under substitution and therefore all minimally imperfect graphs are prime.

We are interested in unavoidable induced subgraphs of prime graphs. Sumner~\cite{Sumner1971,Sumner1973} showed the following easy theorem.
\begin{THM}[Sumner~\cite{Sumner1971,Sumner1973}]
  Every prime graph with at least $3$ vertices has an induced path of length\footnote{A \emph{length} of a path is its number of edges.} $3$.
\end{THM}

What can we say for large prime graphs? 
Before stating our main theorem, we list some basic graph classes.
For a set $X$ of vertices, we say that $X$ is a \emph{clique} if every pair of distinct vertices in $X$ is adjacent and we say $X$ is \emph{independent} if every pair of vertices in $X$ is non-adjacent. 
A graph is a \emph{split graph} if its vertex set can be partitioned into a clique and an independent set.
\begin{itemize}
\item The \emph{half-graph} $H_n$  of height $n$ is a bipartite graph on $2n$ vertices
$a_1$, $a_2$, $\ldots$, $a_n$, $b_1$, $b_2$, $\ldots$, $b_n$
such that $a_i$ is adjacent to $b_j$ if and only if $i\ge j$.
\item The \emph{half split graph} $H'_n$  of height $n$ is the graph obtained from $H_n$ by making $\{b_1,b_2,\ldots,b_n\}$ a clique.
Let $H'_{n,I}$ be the graph obtained from $H'_n$ by adding a new
vertex adjacent to $a_1,a_2,\ldots,a_n$.
Let $H^*_{n}$ be the graph obtained from $H'_n$ by adding a new
vertex adjacent to $a_n$.
\item 
The \emph{thin spider} with $n$ legs is a split graph on $2n$ vertices consisting of an independent set $\{a_1,a_2,\ldots,a_n\}$ and a clique $\{b_1,b_2,\ldots,b_n\}$ such that $a_i$ is adjacent to $b_j$ if and only if $i=j$. 
Thus, the thin spider with $n$ legs is isomorphic to the line graph of $K_{1,n}^{(1)}$.
The \emph{thick spider} with $n$ legs is a split graph on $2n$ vertices consisting of an independent set $\{a_1,a_2,\ldots,a_n\}$ and a clique $\{b_1,b_2,\ldots,b_n\}$ such that $a_i$ is adjacent to $b_j$ if and only if $i\neq j$. 
So, a thick spider is the complement of a thin spider.
A \emph{spider} is a thin spider or a  thick spider.
\end{itemize}

A sequence $v_0,v_1,\ldots,v_n$ of distinct vertices of $G$ is called a \emph{chain} from a set $I\subseteq V(G)$ to $v_n$ if  $n\ge 2$, 
$v_0,v_1\in  I$, $v_2,v_3,\ldots,v_n\notin I$, 
and for $i>0$, $v_{i-1}$ is either the unique neighbor or the unique non-neighbor of $v_i$ in $\{v_0,v_1,\ldots,v_{i-1}\}$.
The \emph{length} of a chain $v_0,v_1,\ldots,v_n$ is $n$.
For example,  a sequence of vertices inducing  a path of length $t$ is a chain of length $t$.

Here is our main theorem.
\begin{figure}
  \centering 
 \tikzstyle{v}=[circle, draw, solid, fill=black, inner sep=0pt, minimum width=3pt]
\subfloat[$K_{1,5}^{(1)}$: the $1$-subdivision of $K_{1,5}$]{
 \begin{tikzpicture}[scale=0.6]
   \node (dummy) at (-1,0) {};\node (dummy2) at (3,0) {};
   \node [v] (v0) at (-1, -3){};
      \foreach \x in {1,...,5} {
        \node [v]  (v\x) at(0,-\x){};        \node [v]  (w\x) at (2,-\x){};
     }
      \foreach \x in {1,...,5} {
      \draw (v0)--(v\x)--(w\x);
      }
    \end{tikzpicture}
}%
\subfloat[$\overline{K_{1,5}^{(1)}}$]{
 \begin{tikzpicture}[scale=0.6]
      \node [v] (v0) at (-1, -3){};
      \foreach \x in {1,...,5} {
        \node [v]  (v\x) at(0,-\x){};        \node [v]  (w\x) at (2,-\x){};
        \draw (v0)--(v\x);
     }
      \foreach \i in {1,2,3,4} {
        \foreach \j in {\i,...,5} {
          \draw[bend right] (v\i) to (v\j);
          \draw[bend left] (w\i) to (w\j);
          }
          }
      \foreach \x in {2,...,5} {      \draw (v1)--(w\x);      }
      \foreach \x in {1,3,4,5} {      \draw (v2)--(w\x);      }
      \foreach \x in {1,2,4,5} {      \draw (v3)--(w\x);      }
      \foreach \x in {1,2,3,5} {      \draw (v4)--(w\x);      }
      \foreach \x in {1,...,4} {      \draw (v5)--(w\x);      }
    \end{tikzpicture}
}%
\subfloat[$L(K_{2,5})$]{
 \begin{tikzpicture}[scale=0.6]
   \node (dummy) at (-1,0) {};\node (dummy2) at (3,0) {};
      \foreach \x in {1,...,5} {
        \node [v]  (v\x) at(0,-\x){};        \node [v]  (w\x) at (2,-\x){};
     }
      \foreach \i in {1,2,3,4} {
        \foreach \j in {\i,...,5} {
          \draw[bend right] (v\i) to (v\j);
          \draw[bend left] (w\i) to (w\j);
          }
          }
      \foreach \x in {1,...,5} {
      \draw (v\x)--(w\x);
      }
    \end{tikzpicture}
}
\subfloat[$\overline{L(K_{2,5})}$]{
 \begin{tikzpicture}[scale=0.6]
   \node (dummy) at (-1,0) {};%
      \foreach \x in {1,...,5} {
        \node [v]  (v\x) at(0,-\x){};        \node [v]  (w\x) at (2,-\x){};
     }
      \foreach \x in {2,...,5} {      \draw (v1)--(w\x);      }
      \foreach \x in {1,3,4,5} {      \draw (v2)--(w\x);      }
      \foreach \x in {1,2,4,5} {      \draw (v3)--(w\x);      }
      \foreach \x in {1,2,3,5} {      \draw (v4)--(w\x);      }
      \foreach \x in {1,...,4} {      \draw (v5)--(w\x);      }
    \end{tikzpicture}
}\\
\subfloat[A  thin spider with $5$ legs]{
 \begin{tikzpicture}[scale=0.6]
   \node (dummy) at (-1,0) {};\node (dummy2) at (3,0) {};
      \foreach \x in {1,...,5} {
        \node [v]  (v\x) at(0,-\x){};        \node [v]  (w\x) at (2,-\x){};
     }
      \foreach \i in {1,2,3,4} {
        \foreach \j in {\i,...,5} {
          \draw[bend left] (w\i) to (w\j);
          }
          }
      \foreach \x in {1,...,5} {
      \draw (v\x)--(w\x);
      }
    \end{tikzpicture}
}%
\subfloat[A  thick spider with $5$ legs]{
 \begin{tikzpicture}[scale=0.6]
   \node (dummy) at (-1,0) {};\node (dummy2) at (3,0) {};
      \foreach \x in {1,...,5} {
        \node [v]  (v\x) at(0,-\x){};        \node [v]  (w\x) at (2,-\x){};
     }
      \foreach \i in {1,2,3,4} {
        \foreach \j in {\i,...,5} {
          \draw[bend left] (w\i) to (w\j);
          }
          }
      \foreach \x in {2,...,5} {      \draw (v1)--(w\x);      }
      \foreach \x in {1,3,4,5} {      \draw (v2)--(w\x);      }
      \foreach \x in {1,2,4,5} {      \draw (v3)--(w\x);      }
      \foreach \x in {1,2,3,5} {      \draw (v4)--(w\x);      }
      \foreach \x in {1,...,4} {      \draw (v5)--(w\x);      }
    \end{tikzpicture}
}\quad
\subfloat[$H_5$: a half-graph of height $5$]{
 \begin{tikzpicture}[scale=0.6]
   \node (dummy) at (-1,0) {};%
      \foreach \x in {1,...,5} {
        \node [v]  (v\x) at(0,-\x){};        \node [v]  (w\x) at (2,-\x){};
     }
     \foreach \x in {1,...,5} {
       \foreach \y in {\x,...,5}{
         \draw (v\x)--(w\y);
       }
     }
    \end{tikzpicture}
}\\
\subfloat[$\overline{H_5}$]{
 \begin{tikzpicture}[scale=0.6]
      \foreach \x in {1,...,5} {
        \node [v]  (v\x) at(0,-\x){};        \node [v]  (w\x) at (2,-\x){};
     }
      \foreach \i in {1,2,3,4} {
        \foreach \j in {\i,...,5} {
          \draw[bend right] (v\i) to (v\j);
          \draw[bend left] (w\i) to (w\j);
          }
       }
       \draw(v2)--(w1);
       \draw(v3)--(w1);
       \draw(v3)--(w2);
       \draw(v4)--(w1);
       \draw(v4)--(w2);
       \draw(v4)--(w3);
       \foreach \y in {1,...,4}{
         \draw (v5)--(w\y);
       }
    \end{tikzpicture}
}%
\subfloat[$H'_{5,I}$ (isomorphic to $\overline{H'_{5,I}}$)]{
 \begin{tikzpicture}[scale=0.6]
   \node [v] (v0) at (-1, -3){};
   \foreach \x in {1,...,5} {
        \node [v]  (v\x) at(0,-\x){};        \node [v]  (w\x) at (2,-\x){};
     }
      \foreach \i in {1,2,3,4} {
        \foreach \j in {\i,...,5} {
          \draw[bend left] (w\i) to (w\j);
          }
          }
     \foreach \x in {1,...,5} {
       \foreach \y in {\x,...,5}{
         \draw (v0)--(v\x)--(w\y);
       }
     }
    \end{tikzpicture}
}%
\subfloat[$H^*_{5}$]{
 \begin{tikzpicture}[scale=0.6]
   \node [v] (v0) at (-1, -1){};
   \foreach \x in {1,...,5} {
        \node [v]  (v\x) at(0,-\x){};        \node [v]  (w\x) at (2,-\x){};
     }
      \foreach \i in {1,2,3,4} {
        \foreach \j in {\i,...,5} {
          \draw[bend left] (w\i) to (w\j);
          }
          }
     \foreach \x in {1,...,5} {
       \foreach \y in {\x,...,5}{
         \draw (v\x)--(w\y);
       }
     }
     \draw (v0)--(v1);
    \end{tikzpicture}
}\quad
 \subfloat[$\overline{H^*_{5}}$]{
  \begin{tikzpicture}[scale=0.6]
   \node [v] (v0) at (3, -2.5){};
   \foreach \x in {1,...,5} {
        \node [v]  (v\x) at(0,-\x){};        \node [v]  (w\x) at (2,-\x){};
     }
      \foreach \i in {1,2,3,4} {
        \foreach \j in {\i,...,5} {
          \draw[bend left] (w\i) to (w\j);
          }
          }
     \foreach \x in {1,...,5} {
       \foreach \y in {\x,...,5}{
         \draw (v\x)--(w\y);
       }
     }
     \foreach \x in {1,...,4} {
       \draw (v0)--(w\x);
     }
  \end{tikzpicture}
 }
  \caption{Typical prime graphs}
  \label{fig:unavoid}
\end{figure}
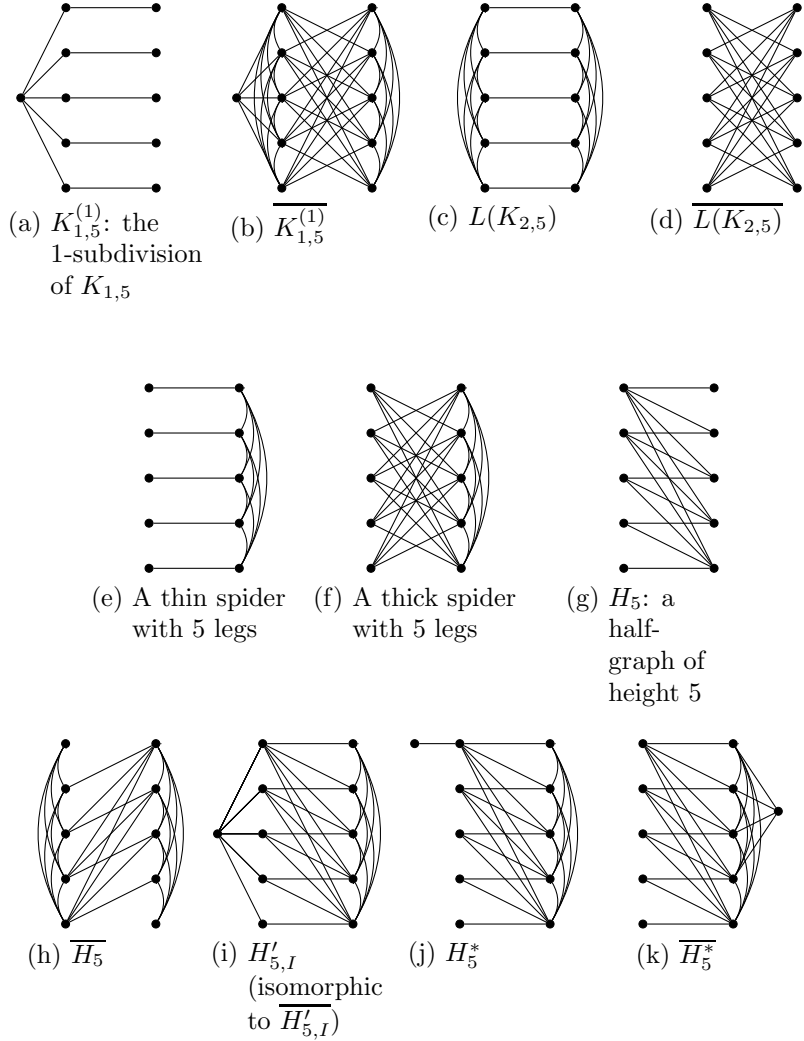
\begin{figure}
  \centering
 \tikzstyle{v}=[circle, draw, solid, fill=black, inner sep=0pt, minimum width=3pt]

  \subfloat[$5K_2$]{
    \begin{tikzpicture}[scale=0.6]
      \foreach \x in {1,...,5} {
        \node [v]  (v\x) at(0,-\x){};        \node [v]  (w\x) at (2,-\x){};
      }
      \foreach \x in {1,...,5} {
        \draw (v\x)--(w\x);
      }
    \end{tikzpicture}
  }\MYSPACE
\subfloat[$H_5'$: the half split graph of height $5$]{ 
 \begin{tikzpicture}[scale=0.6]
   \node (dummy) at (-2,0) {};\node (dummy2) at (4,0) {};
   \foreach \x in {1,...,5} {
        \node [v]  (v\x) at(0,-\x){};        \node [v]  (w\x) at (2,-\x){};
     }
      \foreach \i in {1,2,3,4} {
        \foreach \j in {\i,...,5} {
          \draw[bend left] (w\i) to (w\j);
          }
          }
     \foreach \x in {1,...,5} {
       \foreach \y in {\x,...,5}{
         \draw (v\x)--(w\y);
       }
     }
    \end{tikzpicture}
}%

  \caption{Some non-prime graphs}
  \label{fig:other}
\end{figure}
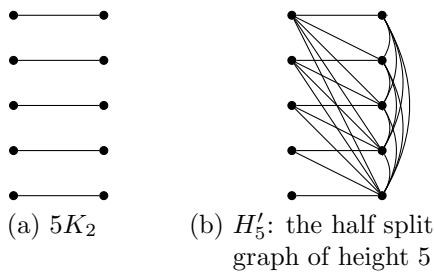

\begin{THM}\label{thm:main}
  For every integer $n\ge3$, there exists $N$ such that every prime graph with at least $N$ vertices contains  one of the following graphs or their complements as an induced subgraph.
  \begin{enumerate}
  \item The $1$-subdivision of $K_{1,n}$ (denoted by $K_{1,n}^{(1)}$).
  \item The line graph of $K_{2,n}$.
  \item The thin spider with $n$ legs.
  \item The half-graph of height $n$.
  \item The graph $H'_{n,I}$.
  \item The graph $H^*_{n}$.
  \item A prime graph induced by a chain of length $n$.
  \end{enumerate}
\end{THM}
Note that all graphs in (1)--(7) are prime.
It is straightforward to prove that 
none of the graphs in the list can be omitted, by showing that none of the graphs in the list contains a large graph of other types or the complement of large graphs in the list.

\section{Chain}\label{sec:chain}
A path is an object to certify connectedness and a chain can be used analogously to certify primeness. In the next proposition, we will prove that in a graph $G$, there  is a chain from every set of two vertices to every other vertex if and only if $G$ is prime.
We say that for a set $X$ of vertices, a vertex $v\notin X$ is \emph{mixed} on $X$ if $v$ has both a neighbor and a non-neighbor in $X$.
We remark that Zverovich~\cite{Zverovich2003} introduced a similar notion called a pseudopath. 

\begin{PROP}\label{prop:chain}
  Let $I\subseteq V(G)$ be a set with at least two vertices.
  and let $v$ be a vertex in $V(G)\setminus I$.
  Then $G$ has a chain from
  $I$ to $v$ if and only if all homogeneous sets containing
  $I$ as a subset contain $v$.
\end{PROP}

\begin{proof}
  For the forward direction, let $X$ be a set such that $I\subseteq X$ and $v\notin X$. Let $v_0,v_1,\ldots,v_n=v$ be a chain. Let $i$ be the minimum positive integer such that $v_i\notin X$. Then $i\ge 2$ because $\abs{I}\ge 2$.
  Since $v_i$ is mixed on $\{v_0,v_1,\ldots,v_{i-1}\}\subseteq X$, $X$ cannot be homogeneous.

  For the backward direction, let us construct an auxiliary digraph $H$
  on $\{\eta\}\cup (V(G)\setminus I)$. 
  For a vertex $w\in V(G)\setminus I$, $(\eta,w)$ is an arc of $H$ if and only if 
  $w$ is mixed on $I$.
  For two distinct vertices $x,y\in V(G)\setminus I$, 
  $(x,y)$ is an arc of $H$ if and only if 
  $y$ is not mixed on $I$ 
  and $y$ is mixed on $I\cup \{x\}$.

  Suppose that all homogeneous sets containing $I$ as a subset contain $v$ and 
  $v$ cannot be reached from $\eta$ by a directed path in $H$.
  Let $Z$ be the set of vertices that can be reached from $\eta$ by a directed path in $H$. If $y\in V(G)\setminus Z$, then 
  $y$ is not mixed on $I$, 
  as $(\eta,y)$ is not an arc of $H$.
  Furthermore, for each $x\in Z\setminus\{\eta\}$, as $(x,y)$ is not an arc of $H$, $y$ is not mixed on $I\cup \{x\}$.
  Therefore, $Z'=(Z\setminus \{\eta\})\cup I$ is homogeneous in $G$, contradicting our assumption that there is no homogeneous set containing $I$ but not containing $v$.

  Thus, $H$ has a shortest directed path $P$ from $\eta$ to $v$ with vertices $\eta$, $w_1$, $w_2$, $\ldots$, $w_n=v$ in order. Let $v_0,v_1$ be a neighbor and a non-neighbor of $w_1$ in $I$, respectively. Such a choice exists because $w_1$ is mixed on $I$. It follows easily that the sequence \[v_0,v_1,w_1,w_2,\ldots,w_n\] is a chain from $I$ to $v$.
\end{proof}
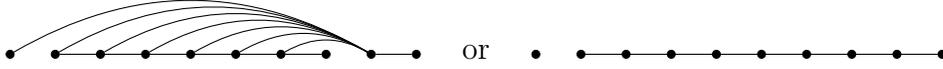
\begin{figure}
  \centering
  \tikzstyle{v}=[circle, draw, solid, fill=black, inner sep=0pt, minimum width=3pt]
  \begin{tikzpicture}[scale=0.6]
  \foreach \x in {0,...,9} {
    \node [v] (v\x) at (\x,0) {};
  }
  \draw (v1) to (v7);
  \draw (v8) to (v9);
  \foreach \y in {0,...,6} {
    \draw [bend right] (v8) to (v\y);
  }
  \end{tikzpicture} \quad or \quad
  \begin{tikzpicture}[scale=0.6]
  \foreach \x in {0,...,9} {
    \node [v] (v\x) at (\x,0) {};
  }
  \draw (v1) to (v9);
  \end{tikzpicture}
  \caption{Two non-prime graphs induced by chains}
  \label{fig:nonprimechain}
\end{figure}

Not every chain induces a prime graph. For instance, the graphs in Figure~\ref{fig:nonprimechain} are induced by chains but are not prime. In the following, we characterize which chains induce prime subgraphs.

\begin{PROP}\label{prop:primechain}
  Let $t\ge 3$ and $C=v_0,v_1,\ldots,v_t$ be a chain of length $t$ in $G$. Then $C$ induces a prime subgraph of $G$ if and only if 
  each of $v_0$ and $v_1$ have a neighbor other than $v_{t-1}$ and a non-neighbor other than $v_{t-1}$ in the chain.
\end{PROP}
This proposition implies that if a graph $G$ is induced by a chain and is not prime, then $G$ is isomorphic to one of the graphs in Figure~\ref{fig:nonprimechain} and their complements.
\begin{proof}
  We may assume that $V(G)=\{v_0,v_1,\ldots,v_{t}\}$.
  Note that we may swap $v_0$ and $v_1$ and still obtain a chain.

  Let us first prove the forward direction. Suppose that $G$ is prime. By taking the complement and swapping $v_0$ and $v_1$ if necessary, we may assume that $v_0$ has no neighbor other than $v_{t-1}$.
If $v_0$ is non-adjacent to $v_{t-1}$, then 
$\{v_1,v_2,\ldots,v_{t}\}$ is homogeneous.
Therefore
$v_0$ is adjacent to $v_{t-1}$. It follows that $v_{t-1}$ is the unique neighbor of $v_t$ and so $\{v_0,v_{t}\}$ is homogeneous.
This contradicts our assumption that $G$ is prime.

  For the backward direction, suppose that each of $v_0$ and $v_1$ have a neighbor other than $v_{t-1}$ and a non-neighbor other than $v_{t-1}$ and furthermore $G$ has  a homogeneous set $X$. By Proposition~\ref{prop:chain}, either $v_0\notin X$ or $v_1\notin X$, because otherwise $V(G)\subseteq X$. We may assume that $v_0\notin X$ by swapping $v_0$ and $v_1$ if necessary.
  
    We may assume that $v_t$ is non-adjacent to $v_0$ by taking the complement graph if needed. Thus $v_{t-1}$ is the unique neighbor of $v_t$.

    For $0\le i<j<t$, the sequence $v_i,v_j,v_{j+1},v_{j+2},\ldots,v_t$ is a chain and by Proposition~\ref{prop:chain}, if $v_i,v_j\in X$, then $v_{j+1},v_{j+2},\ldots,v_t\in X$. Since $\abs{X}\ge 2$, $v_t\in X$.

    We claim that for each $a\in \{1,2,3,\ldots,t-2\}$, if $v_a\notin X$, then $v_{a+1}\notin X$. Suppose not. Then $v_0,v_{a}$ are non-adjacent to $v_t\in X$ and so they are non-adjacent to $v_{a+1}\in X$, contradicting the definition of a chain. 
    This implies that $v_1\in X$ because $\abs{X}\ge 2$. 

    If $v_2\in X$, then $v_{3},v_{4},\ldots,v_t\in X$ because $v_1\in X$. Since $v_0$ is non-adjacent to $v_t$, no neighbor of $v_0$ is in $X$, contradicting our assumption that $v_0$ has a neighbor other than $v_{t-1}$.
    Thus, $v_2\notin X$ and so $v_2,v_3,\ldots,v_{t-1}\notin X$.  

    So $X=\{v_1,v_t\}$.
    However, $v_{t-1}$ is the unique neighbor of $v_t$ and $v_1$ has a neighbor other than $v_{t-1}$ and therefore $X$ is not homogeneous.
\end{proof}

\begin{COR}\label{cor:prime}
  Let $t> 3$.
  Every chain of length $t$ contains a chain of length $t-1$ inducing a prime subgraph.
\end{COR}
\begin{proof}
  Let $v_0,v_1,\ldots,v_t$ be a chain and let $G$ be the graph induced on $\{v_0,v_1,\ldots,v_t\}$. By taking the complement and swapping $v_0$ and $v_1$, we may assume that $v_t$ is complete to $\{v_0,v_1,v_2\}$ and $v_0$ is adjacent to $v_2$. Now consider the  chain $v_1,v_2,\ldots,v_t$. We know that $v_1$ and $v_2$ are non-adjacent and $v_t$ is complete to $\{v_1,v_2\}$ which implies that $v_1,v_2,\ldots,v_t$ is a chain inducing a prime subgraph by Proposition~\ref{prop:primechain}.
\end{proof}

\section{Prime graphs containing a large independent set}\label{sec:first}

If a prime graph $G$ is very large, then it will have a large independent set or a large clique by Ramsey's Theorem. If it has a large clique, then we may take the complement graph and assume that $G$ has a large independent set, because the complement of $G$ is also prime. In this section, we apply Ramsey's Theorem to extract some induced subgraphs that are not necessarily prime. Later sections will be devoted to growing these subgraphs into prime subgraphs.

\begin{PROP}\label{prop:stable}
  For all integers $n,n_1,n_2>0$,there exists $N=f(n,n_1,n_2)$ such that 
  every prime graph with an $N$-vertex independent set 
  contains an induced subgraph 
  isomorphic to 
  \begin{enumerate}
  \item  a  spider with $n$ legs,
  \item $\overline{L(K_{2,n})}$,
  \item the half-graph of height $n$,  
  \item the disjoint union of $n_1$ copies of $K_2$, denoted by $n_1K_2$, or
  \item the half split graph of height $n_2$.
  \end{enumerate}
\end{PROP}
\begin{proof}
  We say that $(A,X,Y)$ is a \emph{regular} triple if 
  $A$, $X$, $Y$ are disjoint subsets of vertices such that
  $A\cup X$ is independent, $\abs{X}=\abs{Y}$, and 
  $X$ and $Y$ can be ordered 
  $X=\{x_1,x_2,\ldots,x_n\}$,
  $Y=\{y_1,y_2,\ldots,y_n\}$ such that for all $i$,
  the vertex $y_i$ is either 
  \begin{enumerate}
  \item adjacent to $x_i$ and anti-complete to 
    $\{x_{i+1},x_{i+2},\ldots,x_{n}\}\cup A$, or
  \item non-adjacent to $x_i$ and complete to
    $\{x_{i+1},x_{i+2},\ldots,x_{n}\}\cup A$.
  \end{enumerate}
  (We allow $y_i$ to be adjacent to $x_j$ if $j<i$.)

  We claim that if $(A,X,Y)$ is a regular triple of a prime graph $G$ with $1<\abs{A}<\abs{V(G)}$,
  then there is a regular triple $(A',X',Y')$ such that $A'\subseteq A$, $\abs{A'}\ge
  \abs{A}/2$ and $\abs{X'}=\abs{Y'}=\abs{X}+1=\abs{Y}+1$.
  Since $A$ is not homogeneous, there is a vertex $y$ such that $y$
  has both a neighbor and a non-neighbor in $A$. 
  Since $A\cup X$ is independent, $y$ does not belong to $X$. Since each $y_i$ is complete or anti-complete to $A$, $y$ is not in $ Y$.
  Let $A_y$ be the set of all neighbors of $y$ in $A$.
  If $\abs{A_y}\ge \abs{A}/2$, then $(A_y,X\cup \{x\},Y\cup \{y\})$
  is a regular triple for a non-neighbor $x$ of $y$ in $A$.
  If $\abs{A_y}< \abs{A}/2$, then 
  $(A\setminus A_y,X\cup \{x\},Y\cup \{y\})$
  is a regular triple for a neighbor $x$ of $y$ in $A$. This proves the claim.

  \medskip
  Now let us use the above claim to prove this proposition.
  Let $m$ be the Ramsey number $R(n_1+n,2n-1,n+n_2,n+n_2-1)$
  and $N=2^{m+1}$. 
  (The Ramsey number $R(n_1,n_2,\ldots,n_k)$ is the minimum $c$ such that for every $k$-edge-coloring of the complete graph $K_c$, there is $i\in \{1,2,\ldots,k\}$ such that it contains a complete subgraph on $n_i$ vertices in which all edges are colored by  $i$.)

  Let $S$ be an independent set of size $N$.
  Then $(S,\emptyset,\emptyset)$ is a regular triple.
  By the claim, there exists a regular triple $(A,X_0,Y_0)$ with
  $\abs{X_0}=\abs{Y_0}= m $.
  Let $X_0=\{x_1,x_2,\ldots,x_m\}$ and $Y_0=\{y_1,y_2,\ldots,y_m\}$ be ordered according to the definition of a regular triple.
  
  Let $K_m$ be the complete graph on the vertex set $\{1,2,\ldots,m\}$.
  Let us color the edges of $K_m$ as follows:
  The color of an edge $ij$ ($i<j$) of $K_m$ is defined as $(a,b)\in \{0,1\}^2$ such that
  \begin{align*}
    a=1& \text{ if and only if $x_i$ is adjacent to $y_j$},\\
    b=1& \text{ if and only if $y_i$ is adjacent to $y_j$}.
  \end{align*}
  By Ramsey's Theorem, there exists a subset $I$ of $\{1,2,\ldots,m\}$ such that all edges in $K_m[I]$ have the same color $(a,b)$ and 
  \[
  \abs{I}=
  \begin{cases}
    n_1+n &\text{if }(a,b)=(0,0),\\
    2n-1 &\text{if }(a,b)=(1,0),\\
    n+n_2 &\text{if }(a,b)=(0,1),\\
    n+n_2-1 &\text{if }(a,b)=(1,1).\\
  \end{cases}
  \]
  Let us define the following sets.
  \begin{align*}
  I_1&=\{i\in I: \text{$x_i$ is adjacent to $y_i$} \},\\
      I_2&=I\setminus I_1,\\
    Z_1&=\{x_i: i\in I_1\}\cup \{y_i:i\in I_1\}, \\
                                                   Z_2&=\{x_i: i\in I_2\}\cup\{y_i:i\in I_2\}, \\
    Z_3&=\{x_i: i\in I_2,~ i>j\text{ for some }j\in I_2\} \\
    &\qquad \cup \{y_i:i\in I_2,~ i<j\text{ for some }j\in I_2\}.
  \end{align*}

  If $(a,b)=(0,0)$ and $\abs{I_1}\ge n_1$, then $G[Z_1]$ has an induced subgraph isomorphic to $n_1K_2$.
  If $(a,b)=(0,0)$ and $\abs{I_2}\ge n+1$, then $G[Z_3]$ has an induced subgraph isomorphic to a half-graph of height $n$.

  If $(a,b)=(1,0)$ and $\abs{I_1}\ge n$, then $G[Z_1]$ has an induced subgraph isomorphic to a half-graph of height $n$.
  If $(a,b)=(1,0)$ and $\abs{I_2}\ge n$, then $G[Z_2]$ has an induced subgraph isomorphic to $\overline{L(K_{2,n})}$.

  If $(a,b)=(0,1)$ and $\abs{I_1}\ge n$, then $G[Z_1]$ has an induced subgraph isomorphic to a  thin spider with $n$ legs. 
  If $(a,b)=(0,1)$ and $\abs{I_2}\ge n_2+1$, then $G[Z_3]$ has an induced subgraph isomorphic to a half split graph of height $n_2$.

  If $(a,b)=(1,1)$ and $\abs{I_1}\ge n_2$, then $G[Z_1]$ has an induced subgraph isomorphic to a half split graph of height $n_2$.
  If $(a,b)=(1,1)$ and $\abs{I_2}\ge n$, then $G[Z_2]$ has an induced subgraph isomorphic to a  thick spider with $n$ legs.
\end{proof}

\section{Making use of a big induced matching}\label{sec:matching}
Outcome (4) of Proposition~\ref{prop:stable} states that the graph contains a large induced matching. 
By Corollary~\ref{cor:prime}, we only need to consider graphs that contain no long chain.
In the next proposition, we deal with the case when a prime graph has a large induced matching and no long chain.

\begin{PROP}\label{prop:matching}
  Let $t\ge 2$ and $n$, $n'$ be positive integers.  Let $h(n,n',2)=n$ and 
  \[    h(n,n',i)=(n-1)R(n,n,n,n,n,n,n,n',n',h(n,n',i-1))+1\]
  for an integer $i>2$.
  Let $v$ be a vertex of a graph $G$
  and let $M$ be an induced matching of $G$ consisting of $h(n,n',t)$ edges not incident with $v$.
  If for each edge $e=xy$  in $M$, there is a chain of length at most $t$ from  $\{x,y\}$ to $v$,  then $G$ has an induced subgraph isomorphic to one of the following:
  \begin{enumerate}
  \item $K_{1,n}^{(1)}$, 
  \item the half-graph of height $n$, 
  \item $\overline{L(K_{2,n})}$,
  \item a spider with $n$ legs, or
  \item the half split graph of height $n'$.
  \end{enumerate}
\end{PROP}
\begin{proof}
  We proceed by induction on $t$. If $t=2$, then trivially $V(M)\cup\{v\}$ induces a subgraph of $G$ isomorphic to $K_{1,n}^{(1)}$. So we may assume that $t>2$.

  We will find an induced matching $M'$ consisting of $h(n,n',t-1)$ edges not incident with $v$ such that $G$ has a chain of length at most $t-1$ from each edge in $M'$ to $v$. For an edge $e$ in $M$, let $C_e$ be a chain of length at most $t$ from $e$ to $v$.

  If there are $n$ chains of length $2$ in $\{C_e:e\in M\}$, then we obtain an induced subgraph isomorphic to $K_{1,n}^{(1)}$. Thus, we may assume that there exists a subset $M_1$ of $M$ such that the length of $C_e$ is larger than $2$ for all $e\in M_1$ and $\abs{M_1}\ge \abs{M}-n+1$.   We say that a vertex $w$ is \emph{mixed} on an edge $e$ if $w$ is not an end of $e$ and $w$ is adjacent to exactly one end of $e$.   
 Note that the first and second vertices of $C_e$ are the ends of $e$ and the third vertex of $C_e$ is mixed on $e$.

  If a vertex $w$ is the third vertex of at least $n$ chains in $\{C_e:e\in M_1\}$, then the vertex $w$ with those $n$ edges in $M_1$ induces a subgraph isomorphic to $K_{1,n}^{(1)}$. Thus we may choose a subset $M_2$ of $M_1$ such that $\abs{M_2}\ge \lceil \frac1{n-1} \abs{M_1}\rceil$ and the third vertices of chains in $\{C_e:e\in M_2\}$ are distinct. Since $M_2$ is an induced matching, no third vertex of $C_e$ is an end of some edge in $M_2$.

  Let $m=\abs{M_2}$. Let $e_1,e_2,\ldots,e_m$ be the edges in $M_2$.  For each edge $e_i$, let $z_i$ be the third vertex in $C_{e_i}$ mixed on~$e_i$. Let $x_i, y_i$ be  the ends of $e_i$ such that $y_i$ is adjacent to $z_i$.

  Let us construct an edge-coloring of $K_m$ on the vertex set $\{1,2,\ldots,m\}$ as follows. For $1\le i<j\le m$, we color the edge $ij$ of $K_m$ by one of the $10$ colors $(a,b,c)\in \{0,1\}^3\cup \{(2,2,2),(3,3,3)\}$ depending on the adjacencies of the pairs $(z_i,z_j)$, $(z_i,y_j)$, $(y_i,z_j)$.
We set $(a,b,c)=(2,2,2)$ if $z_i$ is mixed on $e_j$, and $(a,b,c)=(3,3,3)$ if $z_i$ is not mixed on $e_j$ and $z_j$ is mixed on $e_i$.  If $z_i$ is not mixed on $e_j$ and $z_j$ is not mixed on $e_i$, then we assign colors $(a,b,c)\in \{0,1\}^3$ as follows;
we set $a=1$ if $z_i$, $z_j$ are adjacent, $b=1$ if $z_i$ and $y_j$ are adjacent, and $c=1$ if $y_i$, $z_j$ are adjacent.

  We apply Ramsey's theorem to $K_m$ and obtain a subset $I$ of $V(K_m)$ such that all edges in $K_m[I]$ have the same color $(a,b,c)$ and 
\[\abs{I}=
\begin{cases}
  h(n,n',t-1) &\text{if }(a,b,c)=(0,0,0),\\
  n' & \text{if }(a,b,c)\in\{(1,1,0),(1,0,1)\},\\
  n &\text{otherwise.}
\end{cases}
\]

  Let $G_1$ be the subgraph of $G$ induced on $x_i,z_i$ for all $i\in I$. Let $G_2$ be the subgraph of $G$ induced on $y_i,z_i$ for all $i\in I$.
  
  If $(a,b,c)=(2,2,2)$ or $(3,3,3)$, then the subgraph of $G$ induced on $x_i,y_i$ for all $i\in I$ and one vertex $z_j$ for the maximum or minimum  $j$  in $I$ is isomorphic to $K_{1,n}^{(1)}$.

  If $(a,b,c)=(1,0,0)$, then $G_2$ is a  thin spider with $n$ legs. If $(a,b,c)=(1,1,1)$, then $G_1$ is a  thick spider with $n$ legs. If $(a,b,c)=(1,1,0)$ or $(1,0,1)$, then $G_2$ is a half split graph of height $n'$. If $(a,b,c)=(0,1,0)$ or $(0,0,1)$, then $G_2$ is a half-graph of height $n$. If $(a,b,c)=(0,1,1)$, then $G_1$ is isomorphic to $\overline{L(K_{2,n})}$.

  If $(a,b,c)=(0,0,0)$, then let $M'=\{y_iz_i: i\in I\}$. Clearly $M'$ is an induced matching of size $h(n,n',t-1)$ and $C_{x_iy_i}\setminus\{x_i\}$ is a chain of length at most $t-1$ from $\{y_i,z_i\}$ to $v$ for each $e\in M'$. Since $C_e$ has at least $4$ vertices for each $e\in M_1$, no edge of $M'$ is incident with $v$. We deduce the conclusion by applying the induction hypothesis to $M'$.
\end{proof}

\section{Prime graphs containing a large half split graph}\label{sec:halfsplit}

Since the graphs in outcomes (1)--(3) of Proposition~\ref{prop:stable} are prime and we have just dealt with outcome (4), it remains to handle outcome (5), the situation when a prime graph contains a large half split graph as an induced subgraph. We do that in this section.

\begin{PROP}\label{prop:half}
  For every positive integer $n$, there exists $N=g(n)$ such that 
  every prime graph having a half split graph of height at least~$N$  as an induced subgraph contains a chain of length~$n+1$ or an induced subgraph 
  isomorphic to one of  $H'_{n,I}$, $H^*_{n}$, and $\overline{H^*_{n}}$.
\end{PROP}
\begin{proof}
  Let $N=4^{n-2}(n+1)+2(n-2)+1$.
  Let $G$ be a graph having $H_N$ as an induced subgraph.
  Let $a_1,a_2,\ldots,a_N,b_1,b_2,\ldots,b_N$ be the vertices of a
  half split subgraph $H'_N$ such that $\{a_1,a_2,\ldots,a_N\}$ is independent and $\{b_1,b_2,\ldots,b_N\}$ is a clique and $a_i$ is adjacent to $b_j$ if and only if $i\ge j$.

  Suppose that $G$ has no chain of length~$n+1$ and no induced subgraph isomorphic to $H'_{n,I}, H^*_{n}$, or $\overline{H^*_{n}}$.
  Let $u_0$, $u_1$, $u_2$, $\ldots$, $u_t$ be a shortest chain from $\{a_N,b_N\}$ to $a_1$ or $b_1$ where $u_0=a_N$ and $u_1=b_N$.

  Let $N_1=N-2(n-2)=4^{n-2}(n+1)+1$. 
  Since $t\le n$,  there exists $1=i_1<i_2<i_3<\ldots<i_{N_1}=N$ such that neither $a_{i_j}$ nor $b_{i_j}$ is in $\{u_2, u_3,\ldots, u_{t-1}\}$ for all $j\in \{1,2,\ldots,N_1\}$.
  The subgraph of $G$ induced on $\{a_{i_1},a_{i_2},\ldots,a_{i_{N_1}},b_{i_1},b_{i_2},\ldots,b_{i_{N_1}}\}$ is isomorphic to $H'_{N_1}$.

  Let $N_2=\lceil N_1/4^{n-2}\rceil=n+2$. By the pigeonhole principle, there is a subsequence $p_1=1<p_2<p_3<\ldots<p_{N_2}=N$ of $i_1,i_2,\ldots,i_{N_1}$ such that 
  none of $u_2,u_3,u_4,\ldots,u_{t-1}$ is mixed on 
  $A=\{a_{p_2},a_{p_3},\ldots,a_{p_{(N_2-1)}}\}$ or mixed on $B=\{b_{p_2},b_{p_3},\ldots,b_{p_{(N_2-1)}}\}$. Note that $\abs{A}=\abs{B}=n$.
  Again, the subgraph of $G$ induced on $\{a_1,a_N,b_1,b_N\}\cup A\cup B$ is isomorphic to $H'_{N_2}$.

  By the construction, both $u_0$ and $u_1$ are anti-complete to $A$ and complete to $B$.
  Let $i$ be the minimum such that $u_i$ is complete to $A$ or anti-complete to $B$. 
  Then such $i$ exists because $u_t\in\{a_1,b_1\}$ is either complete to $A$ or anti-complete to $B$.

  If $u_i$ is complete to $A$ and  anti-complete to $B$, 
  then the subgraph of $G$ induced on $A\cup B\cup \{u_i\}$ is 
  isomorphic to $H'_{n,I}$.

  If $u_i$ is complete to $A\cup B$, then 
  let   $p\in\{u_{i-1},u_{i-2}\}$ be a vertex non-adjacent to $u_i$. 
  By the hypothesis, $p$ is complete to $B$ and
  anti-complete to $A$.
  It follows that the subgraph of $G$ induced on \[(A\cup B\cup \{p,u_i\})\setminus \{b_{p_2}\}\] is isomorphic to $\overline{H^*_{n}}$.

  If $u_i$ is anti-complete to $A\cup B$, 
  then let $q\in \{u_{i-1},u_{i-2}\}$ be a vertex adjacent to $u_{i}$.
  By the hypothesis, $q$ is complete to $B$ and anti-complete to $A$ and therefore 
  the subgraph of $G$ induced on 
  \[(A\cup B\cup  \{q,u_i\})\setminus \{a_{p_{N_2-1}}\}\]
  is isomorphic to $H^*_{n}$.
\end{proof}

\section{Proof of the main theorem}\label{sec:proof}
Now we are ready to put all these things together to prove Theorem~\ref{thm:main}.

\begin{proof}[Proof of Theorem~\ref{thm:main}]
  We may assume that $G$ has no chain of length $n+1$ by Corollary~\ref{cor:prime}.
  Let $f$, $g$, $h$ be functions given by Propositions~\ref{prop:stable}, \ref{prop:half}, and \ref{prop:matching}. Let $m=f(n,h(n,g(n),n),g(n))$ and $N=R(m,m)$. 
  
  Let $G$ be a prime graph with at least $N$ vertices. By Ramsey's Theorem, $G$ has an independent set or a  clique of size $m$. Since $\overline G$ is also prime, we may assume that $G$ has an independent set of size $m$.

  We may assume that  $G$ has no half split graph of height $g(n)$ by Proposition~\ref{prop:half}. 

  By Proposition~\ref{prop:stable}, one of the following holds:
  \begin{enumerate}[(i)]
  \item $G$ has an induced subgraph isomorphic to a spider with $n$ legs.
  \item $G$ has an induced subgraph isomorphic to $\overline{L(K_{2,n})}$.
  \item $G$ has an induced subgraph isomorphic to the half-graph of height~$n$.
  \item $G$ has an induced matching of size $h(n,g(n),n)$.
  \end{enumerate}
  If (i), (ii), or (iii) holds, then we are done. 

  If (iv) holds, then let $M$ be an induced matching of size $h(n,g(n),n)$ and let $v$ be a vertex incident with no edge in $M$. For each edge $e=xy$ of $M$, there is a chain from $\{x,y\}$ to $v$ because $G$ is prime by Proposition~\ref{prop:chain}. Since we assumed that there is no chain of length $n+1$, such chains have length at most $n$.  By Proposition~\ref{prop:matching}, $G$ contains an induced subgraph isomorphic to either $K_{1,n}^{(1)}$, a half-graph of height $n$,  $\overline{L(K_{2,n})}$, or a  spider with $n$ legs. This completes the proof.
\end{proof}

\end{document}